%% file: main.tex
\documentclass[a4paper,11pt]{amsart}

\usepackage[english]{babel} %

\usepackage{lmodern}
\usepackage[T1]{fontenc}

\usepackage{amssymb}
\usepackage{mathtools} %
\usepackage{slashed} %
\usepackage[babel=true]{microtype} %
\usepackage[autostyle=true]{csquotes} %
\usepackage[dvipsnames]{xcolor} %
\usepackage[biblatex=true]{embrac} %
\usepackage{tcolorbox}
\usepackage{mathrsfs}
\usepackage{faktor}

\calclayout

\definecolor{Heather}{RGB}{164, 132, 172}
\usepackage[pdfusetitle,bookmarks,pdfpagelabels]{hyperref} %
\hypersetup{
  colorlinks,
  urlcolor=Periwinkle,
  citecolor=Heather,
  linkcolor=teal,
  breaklinks=true,
  final,
}

\usepackage[
    citestyle=numeric-comp,
    bibstyle=numeric-comp,
    backend=biber,
    url=true,
    doi=true,
    isbn=false,
    giveninits=true,
    maxbibnames=100,
    maxcitenames=10,
    sorting=nyt
]{biblatex} 

\DeclareDelimFormat[textcite]{multinamedelim}{\textendash}
\DeclareDelimAlias[textcite]{finalnamedelim}{multinamedelim}

\DeclareSourcemap{
  \maps{
    \map{ %
        \step[fieldset = month, null]
        \step[fieldset = day, null]
        \step[fieldset = language, null]
    }
    \map{ %
        \step[fieldsource = doi, final]
        \step[fieldset = url, null]
    }
    \map{ %
        \step[fieldsource = eprint, final]
        \step[fieldset = url, null]
    }
  }
}

\usepackage{tikz}
\usepackage{tikz-cd}
\usetikzlibrary{cd} %

\usepackage[shortlabels]{enumitem}
\setlist[enumerate,1]{label=\upshape(\arabic*)}
\newlist{myenumi}{enumerate}{1}
\setlist[myenumi,1]{label=\upshape(\roman*)}
\newlist{myenuma}{enumerate}{1}
\setlist[myenuma,1]{label=\upshape(\alph*)}

\usepackage{thmtools}
\declaretheorem[name=Theorem, numberwithin=section]{theorem}
\declaretheorem[name=Theorem, numbered=no]{theorem*}

\declaretheorem[name=Lemma, numberlike=theorem]{lemma}
\declaretheorem[name=Lemma, numbered=no]{lemma*}

\declaretheorem[name=Proposition, numberlike=theorem]{proposition}
\declaretheorem[name=Definition, numberlike=theorem, style=definition]{definition}
\declaretheorem[name=Conjecture, numberlike=theorem, style=remark]{conjecture}

\declaretheorem[name=Example, numberlike=theorem, style=remark]{example}
\declaretheorem[name=Remark, numberlike=theorem, style=remark]{remark}

\declaretheorem[name=Theorem]{theoremx}
\declaretheorem[name=Corollary, numberlike=theoremx]{corollaryx}

\numberwithin{equation}{section}
\allowdisplaybreaks[1]

\usepackage[noabbrev,capitalise]{cleveref} %
\crefdefaultlabelformat{#2\textup{#1}#3}
\crefname{conjecture}{Conjecture}{Conjectures}
\crefname{mtheorem}{MaybeTheorem}{MaybeTheorem}
\crefname{theoremx}{Theorem}{Theorems}
\crefname{corollaryx}{Corollary}{Corollaries}

\usepackage[]{todonotes}

\makeatletter
\providecommand\@dotsep{5}
\def\listtodoname{List of Todos}
\def\listoftodos{\@starttoc{tdo}\listtodoname}
\makeatother
\NewDocumentEnvironment{mytodoenv}{m}{\hypersetup{hidelinks}\textsf{\textbf{#1:}} }{}

\usepackage{mymacros}

\title[PSC with point singularities]{Positive scalar curvature with point singularities}
\author{Simone Cecchini}
\thanks{This work was supported by a grant from the Simons Foundation (MPS-TSM-00007902, SC)}
\address[Simone Cecchini]{Department of Mathematics, Texas A\&M University, College Station, TX 77843, USA}
\email{\href{mailto:cecchini@tamu.edu}{cecchini@tamu.edu}}
\urladdr{\href{https://simonececchini.org}{simonececchini.org}}

\author{Georg Frenck}
\address[Georg Frenck]{Institut f\"ur Mathematik, Universit\"at Augsburg, Universitätsstraße 14,  86159 Augsburg, Bavaria, Germany}
\email{\href{mailto:math@frenck.net}{math@frenck.net}}
\urladdr{\href{http://frenck.net}{frenck.net}}

\thanks{Co-funded by the Deutsche Forschungsgemeinschaft (DFG, German Research Foundation) – Project numbers
  313840899; %
  523079177; %
  427320536; %
  390685587. %
} %

\author{Rudolf Zeidler}
\thanks{Co-funded by the European Union (ERC Starting Grant 101116001 – COMSCAL). Views and opinions expressed are however those of the author(s) only and do not necessarily reflect those of the European Union or the European Research Council. Neither the European Union nor the granting authority can be held responsible for them.}
\address[Rudolf Zeidler]{University of Münster, Mathematisches Institut, Einsteinstr.\ 62, 48149 Münster, Germany}
\curraddr{Universität Potsdam, Institut für Mathematik, Karl-Liebknecht-Str.\ 24--25, 14476 Potsdam, Germany}
\email{\href{mailto:rudolf.zeidler@uni-potsdam.de}{rudolf.zeidler@uni-potsdam.de}}
\urladdr{\href{https://www.rzeidler.eu}{www.rzeidler.eu}}

\hypersetup{
  pdfauthor={Simone Cecchini, Georg Frenck, Rudolf Zeidler}, %
}
\begin{document}

\begin{abstract}
  We show that in every dimension $n \geq 8$, there exists a smooth closed manifold $M^n$ which does not admit a smooth positive scalar curvature (``psc'') metric, but $M$ admits an $\mathrm{L}^\infty$-metric which is smooth and has psc outside a singular set of codimension $\geq 8$.
  This provides counterexamples to a conjecture of Schoen.
  In fact, there are such examples of arbitrarily high dimension with only single point singularities.
  We also discuss related phenomena on exotic spheres and tori.
  In addition, we provide examples of $\mathrm{L}^\infty$-metrics on $\mathbb{R}^n$ for certain $n \geq 8$ which are smooth and have psc outside the origin, but cannot be smoothly approximated away from the origin by everywhere smooth Riemannian metrics of non-negative scalar curvature.
  This stands in precise contrast to established smoothing results via Ricci–DeTurck flow for singular metrics with stronger regularity assumptions.
  Finally, as a positive result, we describe a $\mathrm{KO}$-theoretic condition which obstructs the existence of $\mathrm{L}^\infty$-metrics that are smooth and of psc outside a finite subset.
  This shows that closed enlargeable spin manifolds do not carry such metrics.
\end{abstract}

\maketitle

\input{introduction.tex}

\input{counterexamples.tex}

\input{paulas_question.tex}

\input{obstructions.tex}

\appendix

\input{exotic.tex}

\printbibliography

\end{document}

%% file: introduction.tex
\section{Introduction}
Studying weak lower curvature bounds in various singular settings has been a core aspect of Riemannian geometry since a long time.
While for spaces with \emph{scalar curvature} bounded below such a theory is not yet fully developed, various approaches with results in different directions have been proposed recently (for instance~\cite{cecchini2023lipschitz,LeeLeFloch,GromovPlateauBilliard,Burkhardt-Guim:PointwiseBounds,Bamler:LowerBound,SormaniScalarIntrinsic,CM:Skeleton,dong2023stabilityeuclidean3spacepositive,LeeNaberNeumayer,shi-tam-manifolds-with-boundary,Miao:Corners}).
In this paper, we investigate scalar curvature lower bounds for Riemannian \(\Lp^\infty\)-metrics in the sense of the following definition.

\begin{definition}[{cf.~\cite[Definition~1.3]{CM:Skeleton}}]\label{defi:Linftymetric}
    Let \(M\) be a smooth manifold.
    A \emph{uniformly Euclidean} metric (or for short \enquote{\emph{\(\Lp^\infty\)-metric}}) is a measurable symmetric \((0,2)\)-tensor \(g\) on \(M\) such that there exists a smooth Riemannian metric \(g_0\) on \(M\) and \(\Lambda > 0\) with
    \[
      \Lambda^{-1} g_0(\xi, \xi) \leq g(\xi,\xi) \leq \Lambda g_0(\xi,\xi)  
    \]
    for almost all \(x \in M\) and all tangent vectors \(\xi \in \T_x M\).
\end{definition}

For uniformly Euclidean metrics which are smooth outside a singular subset of sufficiently high codimension, it makes sense to study the behavior of lower scalar curvature bounds on the regular subset of the metric.
The first general results in this direction have been obtained by \textcite{CM:Skeleton}, where the following conjecture of Schoen was popularized.
\begin{conjecture}[{Schoen, cf.~\cite[Conjecture 1.5]{CM:Skeleton}}] \label{schoen}
    Let \(M\) be a closed manifold that does not support any smooth Riemannian metric of positive scalar curvature (i.e.\ \(M\) is Yamabe non-positive).
    Let \(S \subset M\) be a closed submanifold of codimension \(\geq 3\) and let \(g\) be an \(\Lp^\infty\)-metric on \(M\) which is smooth on \(M \setminus S\) and satisfies \(\scal_g \geq 0\) on \(M \setminus S\).
    Then \(g\) extends smoothly to \(M\) and \(\Ric_g = 0\).
\end{conjecture}

In other words, this conjecture predicts that subsets of codimension at least three are invisible for the purposes of scalar curvature in a quite strong sense.
This differs from the setting of codimension one and two, where it is well known that such statements do not hold without additional regularity assumptions or boundary conditions, compare \cite[\S 8.1--8.2]{CM:Skeleton}.
Likewise, the condition of being uniformly Euclidean cannot be dropped from \cref{schoen}, otherwise suitable Schwarzschild metrics would yield counterexamples with point singularities, compare~\cite[\S 8.3]{CM:Skeleton}.

\Cref{schoen} has attracted considerable interest and it has been confirmed in many specific cases.
For instance, if \(\dim(M) = 3\), it holds in general by \textcite[Corollary 1.6]{CM:Skeleton} (modulo the caveat of \cref{rem:coordinate_singular} below).
In the \(4\)-dimensional case, far-reaching results have been established by \textcite{kazaras2019desingularizing}.
In a closely related direction, a version of Llarull's theorem for metrics with \(\Lp^\infty\) point singularities was recently proved by \textcite{chu2024llarullstheorempuncturedsphere}.
Fill-in problems with \(\Lp^\infty\) point singularities in dimension three were studied by \textcite[\S 4]{Mantoulidis-Miao-Tam:CapacitySingular}.
We also note that related results for conical singularities and the positive mass theorem have been proved by \textcite{dai2023positive,dai2024positive_nonspin}.

In sharp contrast to all of this, we show in this paper that \(\Lp^\infty\)-regularity is in general not enough to retain positive scalar curvature (\enquote{psc}) obstructions on closed manifolds, even assuming smoothness outside a single point.

\begin{theoremx}\label{thm:main_counterexamples}
    For every \(n \geq 8\), there exists a closed smooth manifold \(M\) of dimension \(n\) which does not admit a smooth psc metric, but which admits an \(\Lp^\infty\)-metric which is smooth and of uniformly positive scalar curvature on the complement of a connected submanifold \(S \subset M\)of codimension \(\geq 8\).
    Moreover, the metric we obtain is of regularity \(\SobolevW^{2,p/2}\), for every \(p<\operatorname{codim}(S \subset M)\).
    If \(8 \leq n \equiv 0, 1, 2,\) or \(4 \mod 8\), then there are examples where \(S = \{\ast\}\) is a single point.
    
    In particular, \cref{schoen} is false in every dimension \(n \geq 8\) and there are counterexamples where the singular set has arbitrarily high codimension.
\end{theoremx}

\begin{remark}
    Due to the Sobolev \(\SobolevW^{2,p/2}\)-regularity, in our examples the scalar curvature even exists in the weak sense on all of \(M\) as a uniformly positive \(\Lp^{p/2}\)-function for all \(p < \operatorname{codim}(S \subset M)\).
    In particular, our metrics have positive distributional scalar curvature in the sense of \textcite[Definition~2.1]{LeeLeFloch}.
    In fact, the scalar curvature of these examples blows up like \(t^{-2}\), where \(t\) denotes the distance to the singular set.
\end{remark}

The counterexamples in \cref{thm:main_counterexamples} are obtained via classification results based on surgery theory which allow to construct psc metrics with conical singularities on all high-dimensional simply connected manifolds (see \cref{prop:linfty-on-simply-connected-manifolds}).
To contrast this with the existing positive results concerning \cref{schoen}, note that the proofs usually proceed by first blowing up the singularity to a complete psc metric on the complement of the singular set. 
Next, appropriate topological obstructions are used to exclude the existence of such complete psc metrics, leading to a contradiction.
The last step fails precisely because this is unobstructed on punctured simply connected manifolds in high dimensions.

We also note that the literal statement of \cref{schoen} cannot hold for a simpler reason. The following observation was brought to our attention by Hajo Hein:

\begin{remark}\label{rem:coordinate_singular}
    One can introduce \enquote{coordinate singularities}:
    For instance, fix a point \(p \in \Torus^{n}\) on the torus of dimension \(n \geq 3\).
    Choose a diffeomorphism \(F \colon \Torus^n \setminus \{p\} \to \Torus^n \setminus \{p\}\) which does not extend smoothly to all of \(\Torus^{n}\) but whose derivatives and that of its inverse lie in \(\Lp^\infty(\Torus^n)\). Such a diffeomorphism \(F\) can be constructed locally in a neighborhood around \(p\) as the cone of a generic self-diffeomorphism \(f \colon \Sphere^{n-1} \to \Sphere^{n-1}\) that is isotopic to the identity.
     This then allows \(F\) to be interpolated to be the identity map outside a slightly larger neighborhood so that it can be extended globally on \(\Torus^{n} \setminus \{p\}\).
    Then \(g = F^\ast g_{\mathrm{flat}}\) yields an \(\Lp^\infty\)-metric on \(\Torus^{n}\) which is smooth and flat on \(\Torus^{n} \setminus \{p\}\), but which does not extend smoothly to all of \(\Torus^{n}\).
\end{remark}

In light of the above, it makes sense to separate \cref{schoen} into an \emph{obstruction aspect}, that is, concerning the non-existence of \(\Lp^\infty\)-metrics which are smooth and have psc on \(M \setminus S\), and a \emph{rigidity aspect} concerning extendability of \(\Lp^\infty\)-metrics which are Ricci-flat on \(M \setminus S\).
In fact, \cref{thm:main_counterexamples} shows that the obstruction aspect alone fails dramatically (see also \cref{rem:smooth_structure} below).

To study the rigidity aspect, \cref{rem:coordinate_singular} may suggest a minor clarification by only requiring a smooth extension of the metric after possibly changing the smooth atlas near the singular set.
This possibility is implicitly accounted for in the statements of removable singularity results for Einstein metrics (e.g.\ \cites[Theorem~5.1]{Nakajima:ConstructionCoordinates}[Theorem 3.1]{RemovingPointSingular}) and thus the \(3\)-dimensional positive result of \textcite[Corollary~1.6]{CM:Skeleton} needs to be interpreted in this way.
However, in high dimensions we do not expect this to be enough either because there is no general removable singularity theorem for Ricci-flat metrics with point singularities of only \(\Lp^\infty\) regularity.\footnote{For instance, take a cone over a non-round Einstein metric on a sphere (from e.g.~\cite{Boehm:Einstein}). This yields a singular Ricci-flat \(\Lp^\infty\)-metric on \(\R^n\) with a single point singularity at the origin that is unremovable even up to changing the smooth atlas.}

We also note that some (but not all) of the smooth manifolds \(M\) obtained from the construction in \cref{thm:main_counterexamples} are homeomorphic to manifolds which admit smooth psc metrics even though \(M\) itself does not.
Notably, this applies to the examples with point singularities in dimensions \(8 \leq n \equiv 1,2 \mod 8\).
This also yields more concrete counterexamples to the obstruction part of Schoen's conjecture:
\begin{example}[{see \cref{thm:exotic-spheres-have-l-infty-psc,remark:existence_of_non-trivial_Hitchin}}]
    Every high-dimensional exotic sphere admits an \(\Lp^\infty\)-metric which outside a single point is smoothly isometric to the standard round sphere, but there are exotic spheres which do not admit any globally smooth psc metric.
\end{example}
 
In particular, given an exotic sphere $\Sigma$, there is an $\Lp^\infty$-metric $g$ on $\Sigma$ of \(\scal_g = n(n-1)\) which cannot be extended to a smooth metric, and there is a $1$-Lipschitz-map $(\Sigma,g)\to(\Sphere^n,g_\circ)$ of degree $1$ which cannot be a global isometry in the Riemannian sense.
This shows that a naive statement of Llarull's theorem in the setting of $\Lp^\infty$-metrics cannot hold true and instead one needs a more subtle version, for instance as in \cite[Theorem~1.4]{chu2024llarullstheorempuncturedsphere}.

Similarly, a more severe variant of the phenomenon described in \cref{rem:coordinate_singular} concerning the rigidity part occurs on exotic tori:

\begin{example}[{see \cref{thm:exotic-tori-have-l-infty-flat}}]
    There are exotic tori \(T_\Sigma\) with an \(\Lp^\infty\)-metric which outside a single point is smooth and flat, although \(T_\Sigma\) does not admit any globally smooth metric of non-negative scalar curvature.
\end{example}

\begin{remark}[see~\cref{rem:Novikov}]\label{rem:smooth_structure}
    Despite highlighting the previous examples, we emphasize that the failure of \cref{schoen} does not primarily lie in the smooth structure.
    It also follows from the proof of \cref{thm:main_counterexamples} that in every dimension \(\geq 8\) there exist counterexamples, where \(M\) is not homeomorphic to any smooth manifold which admits a psc metric.
    This includes in particular the examples with point singularities where \(8 \leq n \equiv 0, 4 \mod 8\).
    For example, $M$ could be taken to be the product of at least two $\mathrm{K}3$ surfaces and a torus of arbitrary dimension.
\end{remark}

A more analytic approach to study \cref{schoen}, and lower (scalar) curvature bounds on metrics with low regularity in general, is to regularize them using geometric flows, for a selection of works in this direction see \cite{Bamler:LowerBound,Simon:DeformationsC0,Burkhardt-Guim:PointwiseBounds,burkhardtguim2024smoothinglinftyriemannianmetrics,LammSimon:W22,lee2021continuousmetricsconjectureschoen,ChengLeeTam:Singular,chu2022riccideturckflowroughmetrics}.
All existing results following this line need additional assumptions on the metric beyond being uniformly Euclidean---for instance, additional Sobolev or \(\Ct^0\) regularity  near the singular set or being sufficiently \(\Lp^\infty\)-close to a smooth background metric.
It stands to reason that \cref{thm:main_counterexamples} must impose definite restrictions for such a procedure to work for the scalar curvature of \(\Lp^\infty\)-metrics.

\begin{remark}
    Indeed, \textcite[Theorem~6.2]{chu2022riccideturckflowroughmetrics} showed that uniformly Euclidean metrics of \(\SobolevW^{1,n}\)-regularity can be smoothed via Ricci--DeTurck flow (see also \textcite{LammSimon:W22} who treated the \(\SobolevW^{2,2}\) case in dimension four).
    Since our counterexamples from \cref{thm:main_counterexamples} in the point singularity case are in \(\SobolevW^{2,p/2} \subset \SobolevW^{1,p}\) for all \(p < n\), this means that \(\SobolevW^{1,n}\) is the sharp threshold on the Sobolev scale for such a smoothing result in general.
\end{remark}

In another direction, \textcite{burkhardtguim2024smoothinglinftyriemannianmetrics} recently showed that one may smooth away certain \(\Lp^\infty\)-singularities on \(\R^n\) using Ricci--DeTurck flow while preserving non-negativity of the scalar curvature, provided that the singular metric is sufficiently \(\Lp^\infty\)-close to the Euclidean metric.
Since the counterexamples in \cref{thm:main_counterexamples} are based on high-dimensional simply connected manifolds which do not admit psc, at first glance this may look like a subtle global phenomenon of the underlying closed manifold.
However, as we note in the following theorem, they actually lead to examples of \(\Lp^\infty\)-metrics on \(\R^n\) with a single point singularity which are not smoothable while preserving non-negative scalar curvature.
In particular, this answers~\cite[Question~3]{burkhardtguim2024smoothinglinftyriemannianmetrics}.

\begin{theoremx}\label{thm:paulasquestion}
    For $n\ge8$, $n\equiv0, 1, 2, $ or $4 \mod 8$, there exists a metric $g$ on $\R^n$ which has regularity \(\SobolevW^{2,{p/2}}_\loc\) for each \(p < n\) such that
    \begin{itemize}
        \item \(\Lambda^{-1} g_0 \leq g \leq \Lambda g_0\) a.e.\ for some fixed \(\Lambda > 0\) and \(g_0\) the Euclidean metric,
        \item \(g\) is smooth on \(\R^n \setminus \{0\}\),
        \item $\scal_g \geq \lambda|x|^{-2} > 0$ on $\R^n\setminus \{0\}$ for some \(\lambda > 0\),
    \end{itemize}
     but there exists no family $(g_t)_{t \in (0, \varepsilon)}$ of smooth metrics on $\R^n$ satisfying both
    \begin{myenumi}
        \item $\scal_{g_t}\ge0$ for all \(t \in (0,\varepsilon)\),
        \item $g_t\xrightarrow{\Ct^\infty_{\mathrm{loc}}(\R^n\setminus\{0\})} g$ for $t \searrow 0$.
    \end{myenumi}
\end{theoremx}
\cref{thm:main_counterexamples,thm:paulasquestion} suggest to us that the setting of \(\Lp^\infty\)-metrics with thin singular sets is in general insufficient for studying lower scalar curvature bounds in a local sense.
Instead, the fact that some aspects of \cref{schoen} hold true in certain cases appears to be more of a global topological phenomenon.
Thus, as a companion to our counterexamples which are based on simply connected spin obstructions, we note a stronger global spin obstruction to psc which does enable a positive result in the direction of \cref{schoen} for point singularities.
To formulate it, let $M$ be an $n$-dimensional closed spin manifold with fundamental group $\Gamma$.
Denote by $\alpha(M; 1)\in \KO_n(\R)$ the Hitchin genus~\cite{Hitchin} of $M$, that is, the KO-theoretic index of the $\Cl_n$-linear Dirac operator on $M$.
Let $\Cstar\Gamma$ be the maximal real group $\Cstar$-algebra of $\Gamma$, and let $\alpha(M;\Gamma)\in\KO_n(\Cstar\Gamma)$ be the Rosenberg index of $M$~\cite{Rosenberg:PSCNovikovI,Rosenberg:PSCNovikovII,Rosenberg:PSCNovikovIII}, see also \cite[\S 2]{Zeidler:width-largeness-index} for a brief survey.
The inclusion $\R\hookrightarrow\Cstar\Gamma$, sending $\lambda\in\R$ to \(\lambda\cdot 1_{\Cstar\Gamma}\), induces a map $j\colon \KO_n(\R)\to \KO_n(\Cstar\Gamma)$.
We will use the condition that \(j(\alpha(M;1))\neq\alpha(M;\Gamma)\).
In other words, this means that the Rosenberg index of \(M\) does not come from the trivial group.

\begin{theoremx}\label{thm:Ricci-flatness}
    Let $M$ be an $n$-dimensional closed spin manifold with fundamental group $\Gamma$.
    Suppose that $j(\alpha(M;1))\neq\alpha(M;\Gamma)$.
    Let $S\subset M$ be a finite set, and let $g$ be an \(\Lp^\infty\)-metric on $M$ which is smooth outside of $S$.
    If $\scal_g\geq0$ on $M\setminus S$, then $\Ric_g\equiv0$ on $M\setminus S$.
\end{theoremx}

For example, the condition \(j(\alpha(M;1))\neq\alpha(M;\Gamma)\) is satisfied for enlargeable spin manifolds \cite{HS06,HS07}.
This includes spin manifolds that admit a non-zero degree map to the torus, in particular manifolds of the form \(M = \Torus^n \# N\), where \(N\) is an arbitrary spin manifold.
Also, all \(3\)-manifolds not carrying a psc metric are enlargeable and thus are covered by \cref{thm:Ricci-flatness}, which is the case treated by \textcite[Theorem~1.6]{CM:Skeleton}.
A further application of \cref{thm:Ricci-flatness} is the following.

\begin{corollaryx}
    Let $K$ be a $\mathrm{K}3$-surface, $m\ge0$ and $n\ge1$ such that $4m+n\ge3$. Then any $\Lp^\infty$-metric on $K^m\times \Torus^n$ which is smooth away from a finite set $S$ and satisfies $\scal_g\ge0$ on $(K^m\times \Torus^n)\setminus S$ is Ricci-flat on $(K^m\times \Torus^n)\setminus S$.
\end{corollaryx}

See \cite[Remarks~1.10--1.12]{cecchini_long_neck} for further discussions of examples.
Conversely, the condition from \cref{thm:Ricci-flatness} is never satisfied for simply connected spin-manifolds and so avoids the counterexamples of \cref{thm:main_counterexamples}.
It is essentially the most general such condition currently conceivable in the spin setting (compare~\cite[Conjecture~1.5]{Schick:ICM}).

Shortly after the first version of our present manuscript appeared, \textcite[Theorem 1.9]{wang2024scalarcurvaturerigidityspheres} presented a variant of \cref{thm:Ricci-flatness} involving a more general class of singular subsets $S\subset M$ provided that \(M\) admits a suitable map to the torus.

We note that the proof of \cref{thm:Ricci-flatness} follows the same lines as in the works of \textcite{kazaras2019desingularizing,CM:Skeleton}. 
Our contribution is identifying the correct KO-theoretic obstruction to complete psc metrics which already appeared in work of \textcite{cecchini_long_neck}, see also \textcite[\S 3.2]{CZ-generalized-Callias}.
Moreover, as in the classical setting, the positive mass theorem with \(\Lp^\infty\)-point singularities can be reduced to a corresponding psc obstruction statement for closed manifolds of the form \(\Torus^n \# N\).
Indeed, in \cite[Proof of Corollary~B]{kazaras2019desingularizing}, Kazaras extended Lohkamp's argument~\cite[Theorem~6.1]{Lohkamp-hammocks} to uniformly Euclidean metrics.
Therefore, \cref{thm:Ricci-flatness} directly implies the following.

\begin{corollaryx}
Let $(M,g)$ be an $n$-dimensional complete asymptotically flat Riemannian spin manifold, where $g$ is an \(\Lp^\infty\)-metric which is smooth outside a finite set $S\subset M$. 
If $\scal_g\geq0$ on $M\setminus S$, then the ADM mass of each asymptotically flat end is non-negative.
\end{corollaryx}

The paper is organized as follows.
In \cref{sec:not-so-schoen}, we establish \cref{thm:main_counterexamples} and discuss \(\Lp^\infty\)-metrics on exotic spheres and tori.
\cref{sec:non-smoothability} is devoted to the proof of \cref{thm:paulasquestion}.
Finally, in \cref{sec:obstructions} we prove \cref{thm:Ricci-flatness}.

\subsection*{Acknowledgements}
We thank Arthur Bartels, Paula Burkhardt-Guim, Otis Chodosh, Johannes Ebert, Hans-Joachim Hein, Sven Hirsch and Christos Mantoulidis for insightful conversations.
SC gratefully acknowledges the hospitality of Mathematics Münster during his visit as a Young Research Fellow, which significantly contributed to the development of this paper.

%% file: counterexamples.tex
\section{Counterexamples to Schoen's conjecture}\label{sec:not-so-schoen}
In this section, we prove \cref{thm:main_counterexamples}.
The proof is based on the following classification result which stands in contrast to \cite[Theorem A]{stolz_simply-connected-manifolds-of-psc} and complements \cite[Corollary C]{gromov-lawson_classification-of-simply-connected}.

\begin{proposition}\label{prop:linfty-on-simply-connected-manifolds}
    Every closed simply connected manifold $M$ of dimension \(n \geq 5\) admits an $\Lp^\infty$-metric $g$ which lies in \(\SobolevW^{2,p/2}\) for every \(p < n\), is smooth away from a point $\ast\in M$, and whose scalar curvature is uniformly positive on $M\setminus\{\ast\}$.
\end{proposition}

The idea of the proof is the following: First, we remove an open ball. On this new manifold with boundary, there is a \emph{collared} psc metric $g$, that is, $g=h + \D{t}^2$ near the boundary. 
Then we interpolate this to a conical metric $(t/\lambda)^2 h +\D{t}^2$ which has an $\Lp^\infty$-singularity at $t=0$.
We start with the following lemma.

\begin{lemma}\label{lem:collared_metric}
    Let \(M\) be a closed simply connected spin manifold of dimension \(n \geq 5\).
    Let \(\Ball^n \subset M\) be an embedded open ball and \(M' = M \setminus \Ball^n\).
    Then there exists a collared psc metric \(g\) on \(M'\).
\end{lemma}
\begin{proof}
    First, we perform $2$-surgeries on $M$ away from $\Ball^n$ to obtain $M_2$ which is $2$-connected.
    This is possible as every embedded $2$-sphere has trivial normal bundle because of the spin condition.
    
    We write $M_2'\coloneqq M_2\setminus \Ball^n$, and we observe that the inclusion $\partial M_2'=\Sphere^{n-1}\hookrightarrow M_2'$ is $2$-connected by construction. 
    If $n\ge6$, the surgery theorem of 
    \textcite{gromov-lawson_classification-of-simply-connected,Shoen-on-the-structure-of-manifolds} in its form presented in~\cite[Corollary on p. 181]{gajer_riemannian-metrics-of-psc-on-compact-manifolds-with-boundary} implies that there exists a psc metric $g_2$ on $M_2'$ which is collared, that is, $g_2=h+\D{t}^2$ on a collar $\partial M_2'\times [0,\varepsilon)\subset M_2'$.
    Since the reverse surgery of a $2$-surgery has codimension $3$, we can reverse the surgeries from the beginning to obtain a collared psc metric $g$ on $M' = M\setminus \Ball^n$ which agrees with $g_2$ near $\partial M' = \partial M_2'$.

    If on the other hand $n=5$, $M$ itself already admits a psc metric by \cite[Theorem A]{stolz_simply-connected-manifolds-of-psc}\footnote{This also follows directly from \cite[Theorem B]{gromov-lawson_classification-of-simply-connected} and the fact that $\Omega_5^{\mathrm{spin}}=0$.} which can be deformed to equal a torpedo-metric on a neighbourhood of an embedded disk $\Disk^5 \subset M$ (\cite[Theorem 1.1]{chernysh_on-the-homotopy-type-of}, see also \cite[Theorem 1.2]{ebertfrenck}).
    Removing the interior of this disk will yield a psc metric on $M'$ which is collared.
\end{proof}

\begin{remark}
    The statement of \cref{lem:collared_metric} is also true if $M$ is simply connected and non-spin:
    In this case, \cite[Corollary C]{gromov-lawson_classification-of-simply-connected} yields a psc metric on $M$ which---as in the proof above (using \cites[Theorem 1.1]{chernysh_on-the-homotopy-type-of}[Theorem 1.2]{ebertfrenck})---can be deformed to be a torpedo-metric on an embedded disk $\Disk^n\subset M$.
    Deleting the interior of this disk with the torpedo-metric yields the desired collared metric on the complement.
\end{remark}

\begin{proof}[Proof of \cref{prop:linfty-on-simply-connected-manifolds}]
    If $M$ does not admit a spin structure then by \cite[Corollary C]{gromov-lawson_classification-of-simply-connected} there even exists a smooth positive scalar curvature metric on $M$. Therefore, we may assume that $M$ is spin. 

    In this case let \(g\) be a collared psc metric on \(M' = M \setminus \Ball^n\) provided by \cref{lem:collared_metric}, where \(\Ball^n\) is some embedded ball around the point \(p \in M\).
    Denote the restriction of \(g\) to the boundary \(\partial M'\) by \(h\).
     Then \(h\) has uniformly positive scalar curvature due to compactness.
    Let $C\coloneqq\inf_{x\in\partial M'}\scal_h(x)>0$.
    Now, let $R>1$ and $\varepsilon>0$ be such that there exists a smooth, monotonically decreasing function $\alpha\colon \R\to[0,1]$ (see \cref{fig:aux-and-warping-function}) such that $\alpha|_{(-\infty,-R+\varepsilon]} = 1$ and $\alpha|_{[-\varepsilon,\infty)}=0$ with 
    \[-\frac12\sqrt{\frac{C}{4(n-1)(n-2)}}\le\alpha'(t)\le0\quad\text{ and }\quad|\alpha''(t)|\le\frac{C}{8(n-1)}.\]
    Furthermore, let $\lambda\in\mathbb{R}$ be such that 
    \[\lambda\ge\max\left\{R, 2\sqrt{\frac{4(n-1)(n-2)}{C}}\right\} > 1.\]
    \begin{figure}[ht]
        \resizebox{\textwidth}{!}{
            \begin{tikzpicture}
                \node at (0,0) {\includegraphics[width=1\textwidth, trim=60 100 10 380, clip]{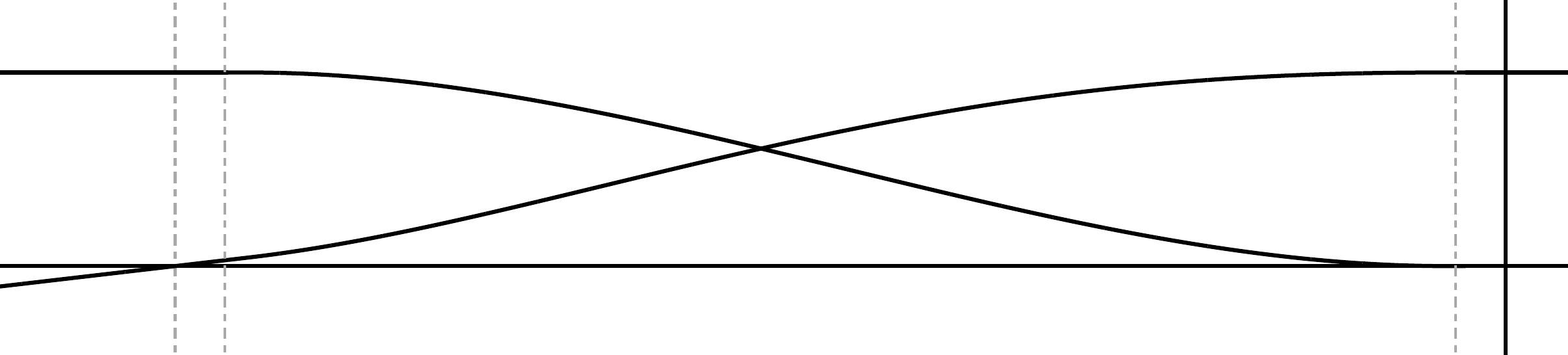}};
                \node at(-5.8,-1.1) {$-R$};
                \node at(-4.3,-1.1) {$-R+\varepsilon$};
                \node at(5.4,-1.1) {$-\varepsilon$};
                \node at(6.8,1.2) {$1$};
                \node at(3,1.3) {$f$};
                \node at(-3,1.3) {$\alpha$}; 
            \end{tikzpicture}
        }
        \caption{The auxiliary function $\alpha\colon\R\to [0,1]$ and the warping function $f\colon \R\to (-\infty,1]$.}\label{fig:aux-and-warping-function}
    \end{figure}

    \noindent We define $f(t)\coloneqq \alpha(t)\frac{t+R}{\lambda} + (1-\alpha(t))$ and consider the conical metric 
    \[
        \gcon =  f(t)^2 h + \D{t}^2
    \]
    on $\Sphere^{n-1}\times (-R,0]$.
    Its scalar curvature is given by
    \begin{align}\label{eq:scal-of-g-conical}
        \scal_{\gcon} &= \frac{1}{f(t)^2}\left(\scal_h - 2(n-1)f(t)f''(t) - (n-1)(n-2)f'(t)^2\right).
    \end{align}
    We note that, by our choices, we have for $t>-R$
    \begin{align}
        f(t) &= -\underbrace{\alpha(t)\left(1-\frac{t+R}{\lambda}\right)}_{\in [0,1)} + 1 \in (0,1]\notag\\
        |f'(t)| &= \left|\alpha'(t)\left(\frac{t+R}{\lambda}-1\right) + \alpha(t)\frac{1}{\lambda}\right|\le |\alpha'(t)|\cdot\underbrace{\left|\frac{t+R}\lambda -1\right|}_{\le1} + \underbrace{|\alpha(t)|}_{\le1}\cdot\left|\frac1\lambda\right|\notag\\
            &\le -\alpha'(t) +\frac1\lambda \le \sqrt{\frac{C}{4(n-1)(n-2)}}\notag\\
            &\Rightarrow (n-1)(n-2)f'(t)^2\le \frac{C}{4}\label{eq:first-derivative-of-f}
    \end{align}
    \begin{align}
        f''(t) &= \alpha''(t)\left(\frac{t+R}{\lambda}-1\right) + \underbrace{\frac{2\alpha'(t)}\lambda}_{\le0}\le |\alpha''(t)|\left(1-\frac{t+R}{\lambda}\right)\notag\\
            &\le |\alpha''(t)| \le \frac{C}{8(n-1)}\notag\\
            &\Rightarrow 2(n-1)f(t)f''(t) \le 2(n-1)f''(t)\le \frac{C}4\label{eq:second-derivative-of-f}.
    \end{align}
    Inserting \cref{eq:first-derivative-of-f} and \cref{eq:second-derivative-of-f} into \cref{eq:scal-of-g-conical}, we obtain:
    \[
        \scal_{\gcon} = \frac{1}{f(t)^2}\Bigl(\scal_h - 2(n-1)f(t)f''(t) - (n-1)(n-2)f'(t)^2 \Bigr)\ge \frac C2>0.
    \]
    Note that $f(-R)=0$ and, since $\gcon = h+\D{t}^2$ for $t>-\varepsilon$, we can extend the metric $g$ on $M'$ by $\gcon$ to a metric $\hat g$ on $M$ which in general has a true singularity at the point \(\ast \in M\) corresponding to $t=-R$.
    A visualization of $\hat g$ can be found in \cref{fig:counterexample}.
    Near the singularity we have $\scal_{\hat g}>\frac C2$, so \(\hat{g}\) has uniformly positive scalar curvature on all of \(M \setminus \{\ast\}\).
    
    It remains to verify that $\hat g$ is uniformly Euclidean and satisfies the desired Sobolev regularity on all of \(M\).
    Since the metric is smooth away from the singularity, it suffices to show that the metric has these properties on some neighborhood including the singularity, where the metric is precisely conical.
    This follows from a direct local computation which we record in \cref{lem:Sobolev_regular} below.
    \begin{figure}
        \centering
        \begin{tikzpicture}
            \node at (0,0) {\reflectbox{
                    \includegraphics[width=.6\textwidth]{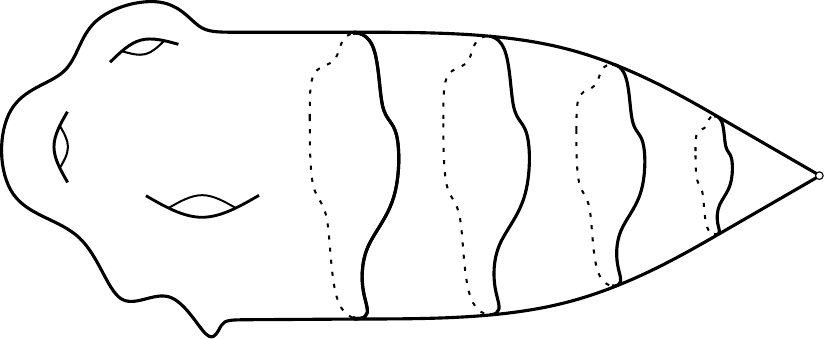}
                }
            };
            \node at (2.3,-2.2) {
                $\underbrace{\quad\qquad\qquad\qquad\qquad}_{M\setminus \Ball^n}$
            };
            \node at (-1.4,-2.2) {
                $\underbrace{\quad\qquad\qquad\qquad\qquad}_{\Sphere^{n-1}\times [-R+\varepsilon,0]}$
            };
            \node at (-5,-1) {
                \footnotesize ${\Sphere^{n-1}\times (-R,-R+\varepsilon]}$
            };
            \node at (2,0.5) {$g$};
            \node[fill=white, minimum height=2em] at (-0.8,0) {$\gcon$};
            \node at (-1.4,1.8) {transition metric};
            \node at (-5,0.8) {$\left(\frac{t+R}{\lambda}\right)^2 h + \D{t}^2$};
        \end{tikzpicture}
        \caption{The metric $\hat g$ on $M$: On $M\setminus \Ball^n$ we have a psc metric $g$ which is collared, then there is a transition region $\Sphere^{n-1}\times [-R+\varepsilon,0]$ and on $\Sphere^{n-1}\times(-R,R+\varepsilon]$ the metric $\gcon$ is the linear scaling of $h$.}
        \label{fig:counterexample}
    \end{figure}
\end{proof}

\begin{lemma}\label{lem:Sobolev_regular}
    Let \(h\) be a Riemannian metric on \(\Sphere^{n-1}\).
    Then the metric \(g=t^2h+\D{}t^2\) on \(\Disk^n\setminus\{0\} \cong \Sphere^{n-1}\times (0,1]\), identified via \(x \mapsto (x/\lvert x \rvert,\lvert x \rvert)\), is uniformly Euclidean on the entire disk \(\Disk^n\) and its coefficients lie in \(\SobolevW^{2,p/2}(\Disk^n) \cap \Lp^\infty(\Disk^n)\) for every $p<n$.
\end{lemma}

\begin{proof}
Since $\Sphere^{n-1}$ is compact, the metric $h$ is \enquote{uniformly round}, that is, there exists $\Lambda_\circ>1$ such that $\Lambda_\circ^{-1} h_\circ\le h \le \Lambda_\circ h_\circ$ for $h_\circ$ the round metric on $\Sphere^{n-1}$.
Then \(g_0 = t^2 h_{\circ} + \D{t}^2\) corresponds to the standard Euclidean metric on \(\Disk^n\) and we have
 \begin{align*}
    \Lambda_{\circ}^{-1} g_{0} \leq \Lambda_{\circ}^{-1} t^2 h_{\circ} + \D{t}^2 \leq t^2 h + \D{t}^2 = g \leq \Lambda_{\circ} t^2 h_{\circ} + \D{t}^2 \leq \Lambda_{\circ} g_{0},
\end{align*}
so \(g\) is uniformly Euclidean on \(\Disk^n\).
    
To verify the Sobolev regularity, let \(\Phi\colon \Disk^n\setminus\{0\}\to \Sphere^{n-1}\times (0,1]\) denote the diffeomorphism sending \(x = (x^1, \dotsc, x^n)\) to \((x/\lvert x \rvert,\lvert x \rvert)\).
Then
\[
    \frac{\partial\Phi}{\partial x^i}
    =\frac{1}{\lvert x \rvert}\left(\frac{\partial}{\partial x^i}-x^i\frac{\mathbf{x}}{\lvert x \rvert^2},x^i\right),
\]
where \(\mathbf{x} = \sum_{k=1}^n x^k \coordvf{x^k}\) is the identity viewed as vector field.
Thus,
\[
    g_{ij}(x)%
    =g_x\left(\frac{\partial\Phi}{\partial x^i},\frac{\partial\Phi}{\partial x^j}\right)
    =h_{x / \lvert x \rvert}\left(\frac{\partial}{\partial x^i}-x^i\frac{\mathbf{x}}{\lvert x \rvert^2},\frac{\partial}{\partial x^j}-x^j\frac{\mathbf{x}}{\lvert x \rvert^2}\right)+\frac{x^ix^j}{\lvert x \rvert^2}
\]
where we used the identification $t=\lvert x \rvert$ under the diffeomorphism $\Phi$.
We deduce that we can write \(g_{ij}(x)=f_{ij}(x/\lvert x \rvert)\) for some smooth functions \(f_{ij} \colon \R^n \to \R\).
Hence, \(g_{ij} \in \bigO(1)\), \(\partial_kg_{ij}\in\bigO(t^{-1})\) and \(\partial_l\partial_kg_{ij}\in\bigO(t^{-2})\) as \(t \to 0\).
This shows that each \(g_{ij}\) is in \(\SobolevW^{2,p/2}(\Disk^n) \cap \Lp^\infty(\Disk^n)\), for every \(p<n\).
\end{proof}

To prove \cref{thm:main_counterexamples}, we need closed simply connected  manifolds which do not support positive scalar curvature.
In dimensions \(\geq 5\) such examples must be spin and they exist whenever the receptacle for the Hitchin \(\alpha\)-invariant~\cite{Hitchin} does not vanish.
The latter is the group \(\KO_n\) which is \(\Z\) if \(n \equiv 0,4 \mod 8\), \(\Z/2\Z\) if \(n \equiv 1,2 \mod 8\) and vanishes otherwise.

\begin{remark}\label{remark:existence_of_non-trivial_Hitchin}
    If \(4 \leq n \equiv 0, 1, 2,\) or \(4 \mod 8\), then there exists a simply connected closed spin \(n\)-manifold \(B\) which does not admit a smooth positive scalar curvature metric because \(\alpha(B;1) \neq 0 \in \KO_n\).
    Indeed, the Atiyah--Bott--Shapiro orientation yields a surjective homomorphism from the spin cobordism group \(\Omega^{\mathrm{spin}}_n \to \KO_{n}\) for \(n \geq 0\) and for \(n \geq 4\) we can use surgery to find simply connected representatives.
    More concretely, for  \(n \equiv 0\) or \(4 \mod 8\) we can just take \(B\) to be a product of \(\mathrm{K}3\) surfaces to obtain a simply connected spin manifold with non-vanishing \(\Ahat\)-genus.
    
    In the cases \(8 \leq n \equiv 1, 2 \mod 8\) one can find exotic spheres with non-vanishing Hitchin genus in \(\Z/2\Z\).
    This follows from classical computations of \textcite{Adams,Milnor:RemarksSpinManifolds}, see also the discussion in \cite[\S4]{Joachim-exotic-spheres}.
\end{remark}

\begin{proof}[Proof of \cref{thm:main_counterexamples}]
    In the cases \(8 \leq n \equiv 0, 1, 2,\) or \(4 \mod 8\) there exists a simply connected spin \(n\)-manifold \(B\) which does not admit a smooth positive scalar curvature as explained in \cref{remark:existence_of_non-trivial_Hitchin} above.
    By \cref{prop:linfty-on-simply-connected-manifolds} however, there exists an $\Lp^\infty$-metric of regularity \(\SobolevW^{2,p/2}\) on $B$ which is smooth away from a point and has uniformly positive scalar curvature.
    This metric is already the desired counterexample.

    The cases of other dimension parities follow from the fact that the product of a smooth metric on another manifold $N$ and an $\Lp^\infty$-metric of regularity \(\SobolevW^{k,p}\) on $B$ which is smooth on $B\setminus \{\ast\}$ yields an $\Lp^\infty$-metric on $B\times N$ which is of regularity \(\SobolevW^{k,p}\) and smooth on  $B \times N \setminus (\{\ast\} \times N)$.
    Therefore, on the one hand, $M\coloneqq B \times \Torus^k$ admits an $\Lp^\infty$-metric which is smooth with uniformly positive scalar curvature outside $S=\{\ast\} \times \Torus^k \subset M$.
    Moreover, such metric has regularity \(\SobolevW^{2,p/2}\) for every \(p<n\).
    Note that the codimension of \(S\) in \(M\) is equal to \(\dim(B) = n \geq 8\).
    On the other hand, $\alpha(B \times \Torus^k; \Z^k) \neq 0$ (see for instance~\cite[Appendix B]{Zeidler:BandWidth}) and so \(M\) does not support a smooth metric of positive scalar curvature.
\end{proof}

\begin{remark}\label{rem:Novikov}
    The \(\Ahat\)-genus is an oriented homeomorphism invariant by a classical result of \textcite[Theorem~1]{Novikov}.
    This means that the examples with point singularities in dimensions \(8 \leq n \equiv 0,4 \mod 8\) constructed in the proof above remain counterexamples even after changing the smooth structure.
    This property persists after taking products with the torus and thus we obtain such counterexamples to Schoen's conjecture in all dimensions \(\geq 8\).
\end{remark}

If we now turn to the case where the smooth structure matters, then we obtain the following more hands-on examples:

\begin{theorem}\label{thm:exotic-spheres-have-l-infty-psc}
    Every exotic sphere $\Sigma$ of dimension \(n \geq 7\) admits an $\Lp^\infty$-metric \(g\) which is smooth away from a point $p\in\Sigma$ together with a homeomorphism \(F \colon \Sigma \to \Sphere^{n}\) which restricts to a diffeomorphism \(\Sigma \setminus \{p\} \xrightarrow{\cong} \Sphere^{n} \setminus \{F(p)\}\) such that \(F^\ast g_{\circ} = g\) on \(\Sigma \setminus \{p\}\), where \(g_{\circ}\) is the round metric.
    In particular, \(g\) has scalar curvature equal to $n(n-1)$ on $\Sigma\setminus\{p\}$.
\end{theorem}

    Since there are exotic spheres with non-vanishing $\alpha$-invariant (see \cref{remark:existence_of_non-trivial_Hitchin}), this again yields  counterexamples to Schoen's conjecture.
    It also demonstrates that Llarull's theorem with point singularities by \textcite{chu2024llarullstheorempuncturedsphere} is sharp, that is, in the situation of \cite[Theorem~1.4]{chu2024llarullstheorempuncturedsphere} one cannot conclude that the metric is smooth on all of \(M\) and \(f\) is a smooth isometry.

\begin{proof}[Proof of \cref{thm:exotic-spheres-have-l-infty-psc}]
    By the proof of the $h$-cobordism theorem \cite{smale_on-the-structure-of-manifolds}, for every exotic sphere $\Sigma$ exists a diffeomorphism $f\colon \Sphere^{n-1}\to \Sphere^{n-1}$ such that $\Sigma = \Disk^n\cup_f \Disk^n$, that is $\Sigma$ is given by gluing two disks along the diffeomorphism $f$.
    After identifying $\Sphere^n = \Disk^n\cup_{\id} \Disk^n$, the Alexander trick yields a homeomorphism $F\colon\Sigma\to \Sphere^n$ given by 
    \[
        F(x)\coloneqq \begin{cases}
            x &\text{ for } x\in \text{left-hand }\Disk^n,\\
            0 &\text{ for } 0=x\in \text{right-hand }\Disk^n,\\
            \norm{x}\cdot f\left(\frac{x}{\norm{x}}\right) &\text{ for } 0\not=x\in \text{right-hand }\Disk^n\setminus\{0\},
        \end{cases}
    \]
    where in the first case we map to the left-hand $\Disk^n$ in $\Sphere^n$ and in the latter two cases we map to the right-hand $\Disk^n$ in $\Sphere^n$.
    Note that $F$ is smooth away from $p=0\in \text{right-hand } \Disk^n\subset \Sigma$.
    Let $g_\circ$ be the round metric on $\Sphere^n$ and let $g\coloneqq F^*g_\circ$.
    The metric $g$ is smooth on $\Sigma\setminus\{p\}$ and has an honest singularity at $p$.
    
    It remains to show that this metric is uniformly Euclidean.
    First, we consider polar coordinates given by the identification $\Phi\colon \Disk^n\setminus\{0\}\cong \Sphere^{n-1}\times(0,1]$, $x\mapsto(x/|x|,|x|)$ as before.
    Its inverse is given by $(q,t)\mapsto t \cdot q$.
    In these coordinates on the right-hand \(\Disk^n\) the homeomorphism $F$ is given by $F_\Phi(q,t)=(f(q),t)$ and the metric $g_\circ$ is of the form $\sin(t)^2h_\circ+\D{t}^2$ for $h_\circ$ the round metric on $\Sphere^{n-1}$.
    Furthermore, we have
    \[
        F_\Phi^*(\sin(t)^2h_\circ+\D{t}^2)=\sin(t)^2 f^\ast h_\circ + \D{t}^2.
    \]
    As in the proof of \cref{lem:Sobolev_regular}, compactness of $\Sphere^{n-1}$ yields a constant $\Lambda\ge1$ such that $\Lambda^{-1}h_\circ \le f^\ast h_\circ \le \Lambda h_\circ$. We obtain
    \[\Lambda^{-1}(\sin(t)^2h_\circ+\D{t}^2)\le \sin(t)^2 f^\ast h_\circ +\D t^2\le \Lambda(\sin(t)^2h_\circ+\D{t}^2)\]
    and hence the metric $g$ is uniformly Euclidean.
\end{proof}

A similar idea also shows that there are exotic tori which admit a flat \(\Lp^\infty\)-metric with a  point singularity even though exotic tori do not admit smooth flat metrics.

\begin{theorem}\label{thm:exotic-tori-have-l-infty-flat}
    Let \(\Sigma\) be an exotic sphere of dimension \(n \geq 7\) and consider the smooth manifold \(T_\Sigma = \Torus^{n} \# \Sigma\).
    Then \(T_\Sigma\) admits an \(\Lp^\infty\)-metric \(g\) which is smooth and flat away from a point \(p \in T_\Sigma\).
    Moreover, if \(\Sigma\) is a non-trivial exotic sphere, then \(T_\Sigma\) does not admit a smooth metric of non-negative scalar curvature.
\end{theorem}
\begin{proof}
    As in (the proof of) \cref{thm:exotic-spheres-have-l-infty-psc}, we write $\Sigma = \Disk^n\cup_f \Disk^n$ for some diffeomorphism $f\colon \Sphere^{n-1}\to \Sphere^{n-1}$.
    Taking a small open ball \(\Ball^n \subset \Torus^{n}\), we can identify \(T_\Sigma = (\Torus^{n} \setminus \Ball^n) \cup_{f} \Disk^{n}\) and \(\Torus^n = (\Torus^n \setminus \Ball^n) \cup_{\id} \Disk^n\) and construct again an Alexander homeomorphism \(F \colon T_\Sigma \to \Torus^n\) given by
    \[
        F(x)\coloneqq \begin{cases}
            x &\text{ for } x \in \Torus^n \setminus \Ball^n\\
            0 &\text{ for } 0=x\in \Disk^n\\
            \norm{x}\cdot f\left(\frac{x}{\norm{x}}\right) &\text{ for } 0\not=x\in \Disk^n\setminus\{0\}.
        \end{cases}
    \]
    We let \(g_{\mathrm{flat}}\) be any flat metric on \(\Torus^n\) and let \(g \coloneqq F^\ast g_{\mathrm{flat}}\).
    As above, this yields the desired \(\Lp^\infty\)-metric on \(T_\Sigma\) which is smooth and flat outside the point \(p \in T_\Sigma\) corresponding to \(0 \in \Disk^n\).
    
    Moreover, if \(\Sigma\) is a non-trivial exotic sphere, then \(T_\Sigma\) is not diffeomorphic to the standard torus by \cref{thm:exotic-tori}.
    However, \(T_\Sigma\) is still an enlargeable spin manifold and thus any smooth non-negative scalar curvature metric on \(T_\Sigma\) must be flat by \cite[Theorem~5.8]{GL1983}.
    But any two flat manifolds with isomorphic fundamental groups are diffeomorphic by Bieberbach's second theorem (see e.g.~\cite[Theorem 5.4]{Charlap:Bieberbach}).
    In particular, \(T_\Sigma\) cannot admit a flat metric and thus also no metric of non-negative scalar curvature.
\end{proof}

%% file: paulas_question.tex
\section{Non-smoothability of \texorpdfstring{\(\Lp^\infty\)}{L∞}-metrics on \texorpdfstring{\(\R^n\)}{ℝⁿ} with positive scalar curvature outside a point}\label{sec:non-smoothability}

Here we prove \cref{thm:paulasquestion} by a similar construction as in the previous section.
\begin{proof}[Proof of \cref{thm:paulasquestion}]
    Let $n$ be as specified in the statement of the theorem.
    Then by \cref{remark:existence_of_non-trivial_Hitchin,lem:collared_metric} there exists a closed \(n\)-manifold \(M\) which does not admit psc, but there exists a psc metric on $g_M$ on $M\setminus \Ball^n$ that is collared near the boundary. 
    In particular, we can arrange that $g_M=h+\D t^2$ near $\partial(M\setminus \Ball^n)=\Sphere^{n-1}$ for some metric $h$ on $\Sphere^{n-1}$ with $\scal_h\ge1$. We define an $\Lp^\infty$-metric on $\R^n$ by
    \[g\coloneqq \left(\frac{t}{\lambda}\right)^2h + \D t^2\]
    where $\lambda\coloneqq \sqrt{2 (n-1)(n-2)}$. 
    As in the proof of \cref{lem:Sobolev_regular}, since any metric on \(\Sphere^{n-1}\) is \enquote{uniformly round} and the Euclidean metric is given by \(g_\circ = t^2 h_{\circ} + \D{t}^2\), this means that there exists a constant \(\Lambda > 1\)  such that \(\Lambda^{-1} g_\circ \leq g \leq \Lambda g_\circ\).
    The metric \(g\) is smooth on \(\R^n \setminus \{0\}\) and the desired \(\SobolevW^{2,p/2}_\loc\)-regularity near \(0\) also follows from \cref{lem:Sobolev_regular}. 
    Its scalar curvature on $\R^n\setminus\{0\}$ is given by 
    \[\scal_g = \left(\frac{\lambda}{t}\right)^2\left(\scal_h - \frac{(n-1)(n-2)}{\lambda^2}\right)\ge\frac12\lambda^2t^{-2}>0.\]
    Furthermore, the induced metric on $\Sphere^{n-1}_{\lambda} = \{t=\lambda\}\subset\R^n$ is precisely given by $h$.
    The mean curvature\footnote{Using the convention that the mean curvature of the unit sphere with respect to the Euclidean metric is equal to \(n-1 > 0\).} $\mean_\lambda$ of $\Sphere^{n-1}_{\lambda}$ satisfies
    \[\mean_\lambda = (n-1) \frac{1}\lambda >0.\]
    
    Let us now proceed by contradiction and assume the existence of the family of smooth metrics $g_t$ with non-negative scalar curvature and converging to \(g\) in $\Ct^\infty_{\loc}(M\setminus\{0\})$.
    Then there exists an $\varepsilon>0$ such that $\scal_{g_\varepsilon}(p)\ge\frac14$ for all $\lambda/2 \leq |p| \leq \lambda$, the metric $h_\varepsilon$ induced by $g_\varepsilon$ on $\Sphere^{n}_\lambda$ satisfies $\scal_{h_\varepsilon}\ge\frac12>0$ and the mean curvature of $\Sphere^n_\lambda$ is $\mean_{\lambda}\ge(n-1)/(2\lambda)$.
    Consider now the restriction of $g_\varepsilon$ to $\Disk^n_\lambda$. 
    It has non-negative scalar curvature overall, positive scalar curvature near the mean convex boundary \(\partial \Disk^n_\lambda\) and the induced metric on the boundary equals $h_\varepsilon$ and has positive scalar curvature.
    Thus, we can assume that $g_\varepsilon$ is collared while retaining non-negative scalar curvature and without changing the restriction to the boundary \(\Sphere_\lambda^{n-1}\) by \cite[Corollary~4.9]{baerhanke_boundary-conditions-for-scalar-curvature}\footnote{Even though it is not stated in loc.cit.\ the deformation principle they prove leaves the metric on the boundary fixed (\cite[Theorem 3.7 (b)]{baerhanke_boundary-conditions-for-scalar-curvature}). In fact the proof of \cite[Theorem 4.8]{baerhanke_boundary-conditions-for-scalar-curvature} also works for a fixed psc metric $h$ on the boundary and then \cite[Theorem 4.5]{baerhanke_boundary-conditions-for-scalar-curvature} with $\mathscr{X}=\{h\}$.}.

    Furthermore, if $\varepsilon$ is small enough, $h_\varepsilon$ and $h$ are isotopic through psc metrics, since the space of psc metrics is an open subspace of the space of all metrics in the (weak) $\Ct^\infty$-topology. 
    Hence, there is a concordance from $h$ to $h_\varepsilon$ by \cite[Lemma 3]{gromov-lawson_classification-of-simply-connected} (see also \cite[Lemma 2.5]{ebertfrenck}), that is, there is a psc metric $G$ on the cylinder $\Sphere^{n-1}\times[0,1]$ which is collared near both boundaries and restricting to $h_\varepsilon+\D t^2$ and $h +\D t^2$ on the respective boundary components. 
    
    Gluing $G$ onto $g_\varepsilon|_{\Disk^n_\lambda}$, we obtain a metric on $\Disk^n$ of non-negative scalar curvature with collared boundary that restricts to $h$ near the boundary.
    Then we can glue this metric onto $g_M$ to obtain a metric on the closed manifold $M$ which has non-negative and non-vanishing scalar curvature.
    But this would mean that \(M\) also has a metric of positive scalar curvature (for instance, run the Ricci flow for a short time), a contradiction to the choice of \(M\).
\end{proof}

\cref{thm:paulasquestion} affirmatively answers \cite[Question 3]{burkhardtguim2024smoothinglinftyriemannianmetrics}. In \emph{loc.\ cit.\ }it is shown that an $\Lp^\infty$-metric $g$ admits a \enquote{smoothing family} $(g_t)_{t>0}$ such as in the statement of \cref{thm:paulasquestion}, provided the metric $g$ is $\Lp^\infty$-close to the Euclidean one.
\textcite[Question 3]{burkhardtguim2024smoothinglinftyriemannianmetrics} then asks if the assumption of $\Lp^\infty$-closeness is necessary.
Indeed, our \cref{thm:paulasquestion} shows that this condition cannot be replaced by merely requiring the metric to be uniformly bilipschitz to the Euclidean metric.

%% file: obstructions.tex
\section{Relative Dirac obstructions} \label{sec:obstructions}
The proof of \cref{thm:Ricci-flatness} is based on the following proposition.

\begin{proposition}\label{prop:KO-obstructions}
    Let $M$ be a closed $n$-dimensional spin manifold with fundamental group $\Gamma$.
    Let $ S$ be a finite subset of $M$, and let $g$ be a uniformly Euclidean metric on $M$ which is smooth on \(M \setminus S\).
    If $\scal_g>0$ on $M\setminus S$, then $j(\alpha(M;1))=\alpha(M;\Gamma)$.
\end{proposition}

\begin{proof}
To simplify the notation, set $\hat M=M\setminus S$ and $\hat g= g|_{\hat M}$.
We will first construct a complete psc metric on $\hat M$ following \cite[Proof of Proposition~6.2]{CM:Skeleton} and \cite[Proof of Theorem~A]{kazaras2019desingularizing}.
Consider the modified conformal Laplacian
\[
    \mathcal L_\sigma\coloneqq-4\frac{n-1}{n-2}\Delta+\sigma,
\]
where $\Delta$ is the Laplace-Beltrami operator for $(\hat M,\hat g)$, and $\sigma$ is a positive function on $M\setminus S$ such that $0<\sigma\leq\min(1,\scal_{\hat g})$.
Let $G$ be the positive distributional solution of the elliptic PDE
\[
    \mathcal L_\sigma G=\delta_{ S}\qquad\textrm{on }M,
\]
where $\delta_{ S}$ is the Dirac delta measure on $S$.
By \cite[Theorem~7.1]{LSW63} (see also~\cite[Appendix A]{ChengLeeTam:Singular} for more details), there exists $c>0$ such that
\begin{equation}\label{eq:complete}
    c^{-1}\dist_g( S,x)^{2-n}\leq G(x) \leq c\dist_g( S,x)^{2-n}\qquad\textrm{on }M\setminus S.
\end{equation}
For $\varepsilon>0$, consider the metric \(g_\varepsilon\coloneqq \left(1+\varepsilon \mathop{G}\right)^\frac{4}{n-2}\hat g\) on $\hat M$.
By \eqref{eq:complete}, $g_\varepsilon$ is complete.
By direct computation,
\[
    \scal_{g_\varepsilon}
    =\left(1+\varepsilon \mathop G\right)^{-\frac{n+2}{n-2}}
    \left(\scal_{\hat g}+\varepsilon\left(\scal_{\hat g}-\sigma\right)\mathop G\right)
    \geq \scal_{\hat g}\left(1+\varepsilon \mathop G\right)^{-\frac{n+2}{n-2}}>0.
\]
In conclusion, we constructed a complete metric of positive scalar curvature on $\hat M$.
By a direct application of \cite[Theorem~C]{cecchini_long_neck}, we conclude that $j(\alpha(M;1))=\alpha(M;\Gamma)$.
\end{proof}

\begin{proof}[Proof of \cref{thm:Ricci-flatness}]
    We proceed by contraposition.
    Suppose that $\Ric_g\not\equiv 0$ on $M\setminus S$.
    By \cite[Section~7, Proof of Theorem~1.7]{CM:Skeleton}, there exists a uniformly Euclidean metric $g_1$ on \(M\) which is smooth \(M \setminus S\) and such that $\scal_{g_1}>0$ on $M\setminus S$.
    Then \cref{prop:KO-obstructions} shows that $j(\alpha(M;1))=\alpha(M;\Gamma)$.
\end{proof}

%% file: exotic.tex
\section{Exotic tori}
Given a smooth manifold $M$, let $\manifoldset(M)$ denote the set of smooth manifold structures on $M$, that is, $\manifoldset(M)$ is the set of smooth manifolds homotopy equivalent to $M$ modulo diffeomorphism. 
We call elements of $\manifoldset(M)$ \emph{exotic} $M$'s, which are alternatively also called \enquote{smooth \emph{fake} \(M\)'s}, see \cite[Definition 18.2]{Luck-Macko_Surgery-theory-Foundations}.
We will focus on the examples \(M = \Sphere^n\) and \(M = \Torus^n\), where all exotic \(M\)'s are known to be homeomorphic to \(M\), see~\cite[Sections~2.1, 12.1 and 19.4]{Luck-Macko_Surgery-theory-Foundations} for detailed references.
In this section we exhibit the following theorem which is folklore (compare for instance~\cite{FarrellJones:expanding_endo_tori}), but it is hard to locate an explicit reference in the literature.
\begin{theorem}\label{thm:exotic-tori}
    If $\Sigma\not=\Sphere^n\in \manifoldset(\Sphere^n)$, then $\Torus^n\#\Sigma\not=\Torus^n\in\manifoldset(\Torus^n)$.
\end{theorem}
The proof follows easily via standard tools in surgery theory.
Start by recalling the relevant notation from the textbook of \textcite{Luck-Macko_Surgery-theory-Foundations}.
For a closed manifold $M$ let $\structureset(M)$ denote the (simple)\footnote{We will only apply this to the sphere and the torus, where every homotopy equivalence is simple.} structure set (\cite[Definition 11.2]{Luck-Macko_Surgery-theory-Foundations}), which consists of (simple) homotopy equivalences $f\colon N\to M$ modulo the following equivalence relation:
\[
    (f_0\colon N_0\to M)\sim(f_1\colon N_1\to M) \quad:\iff\quad \begin{matrix}
        \exists h\colon N_0\to N_1\text{ diffeomorphism such that}\\
        f_1\circ h \text{ is homotopic to } f_0.
    \end{matrix}
\]

By \cite[Sections 17.5 and 18.10]{Luck-Macko_Surgery-theory-Foundations}, we have $\structureset(M)\cong [M,\topo]$ for $n\ge5$ and $M=\Sphere^n$ or $M=\Torus^n$, where $\topo$ is a certain infinite loop space.

\begin{proposition}\label{prop:connected_sum_injective_on_structure_set}
    The map\footnote{The map is defined by choosing a representative \(f \colon \Sigma \to \Sphere^n\) equal to the identity in a disk where the connected sum is to be taken.} 
    \[ 
        \structureset(\Sphere^n) \to \structureset(\Torus^n), \quad [f \colon \Sigma \to \Sphere^n] \mapsto [\id_{\Torus^n} \# f \colon \Torus^n \# \Sigma \to \Torus^n],
    \]
    is injective.
\end{proposition}
\begin{proof}
    Consider the standard cell-decomposition of $\Torus^n$ with $\binom{n}{k}$ many $k$-cells and let $c\colon \Torus^n\to\Sphere^n$ be the degree 1 map collapsing a neighborhood of the $(n-1)$-skeleton to a point.
    Precomposition yields a map
    \[
        [\Sphere^n,\topo]\longrightarrow[\Torus^n,\topo].
    \]
    Note that $c$ maps a disk inside the $n$-cell of $\Torus^n$ to the complement of the base point in $\Sphere^n$.
    Therefore, given an element $[\Sigma \to \Sphere^n] \in \structureset(\Sphere^n)\cong [\Sphere^n,\topo]$, its image under the composition with the above map and $[\Torus^n, \topo] \cong \structureset(\Torus^n)$ is given by $[\Torus^n\#\Sigma \to \Torus^n]$.
    Now, $\topo$ is an infinite loop space, so there exists a space $Y$ such that $\Omega Y\simeq \topo$ and we deduce $[X,\topo] = [X,\Omega Y]=[S X,Y]$ for any space $X$ where $S$ denotes the reduced suspension.\footnote{We deviate from the standard notation $\Sigma$ for the reduced suspension to not confuse it with the exotic sphere $\Sigma$.}
    Since the suspension of the torus decomposes into a wedge-product of spheres $S\Torus^n \cong \Sphere^{n+1}\vee\bigvee_{i}\Sphere^{n_i}$.
    The induced map $Sc\colon S\Torus^n\to S\Sphere^n$ translates to being the identity on $\Sphere^{n+1}$ and collapsing every other factor $\Sphere^{n_i}$ in the above wedge-decomposition to the base point.
    Therefore, the middle map of
    \begin{align*}
        [S\Sphere^n,Y]\cong\left[\Sphere^{n+1},Y\right]\cong\pi_{n+1}(Y)\to
            \pi_{n+1}&(Y)\oplus\bigoplus_{i}\pi_{n_i}(Y)\\
                & \cong  \left[\Sphere^{n+1}\vee\bigvee_{i}\Sphere^{n_i},Y\right]\cong [S\Torus^n, Y]
    \end{align*}
    is given by the inclusion into the first summand.
    Hence, this composition is injective and so is the map $\structureset(\Sphere^n)\to\structureset(\Torus^n$).
\end{proof}

In order to obtain $\manifoldset(M)$ from $\structureset(M)$, we need to divide out the group $\haut(M)$ of homotopy classes of (simple) self-homotopy equivalences of \(M\) (see the discussion around \cite[Definition 11.3]{Luck-Macko_Surgery-theory-Foundations}).
While this is a difficult problem to study in general, for our purposes the observation recorded in the following lemma is enough.

\begin{lemma}\label{lem:haut-action-on-base-point}
    Let \(M\) be a smooth manifold such that every simple self-homotopy equivalence \(M \to M\) is homotopic to a diffeomorphism.
    If \([f \colon N \to M] \neq [\id_M \colon M \to M] \in \structureset(M)\), then \(N\) is not diffeomorphic to \(M\).
\end{lemma}
\begin{proof}
    Since every element of $\haut(M)$ can be represented by a diffeomorphism, the trivial element \([\id_M] \in \structureset(M)\) is a fixed point of the $\haut(M)$-action on \(\structureset(M)\).
    Thus the preimage of \([M] \in \manifoldset(M)\) under the quotient map \(\structureset(M) \to \structureset(M) / \haut(M) = \manifoldset(M)\) consists of the single element \([\id_M]\).
\end{proof}

This is useful, for instance because we have $\haut(\Sphere^n)=\Z/2$.
More precisely, every homotopy self-equivalence of the sphere is homotopic to the identity or the orientation reversing diffeomorphism \((x_1, x_2, \dotsc, x_{n+1}) \mapsto (-x_1, x_2, \dotsc, x_{n+1})\).
For the torus, the situation is slightly more complicated, as we have $\haut(\Torus^n)\cong\GL_n(\Z)$.
This follows from the fact the torus is a $\mathrm{K}(\Z^n,1)$ and homotopy classes of maps $[\Torus^n,\Torus^n]$ are hence determined by their induced maps on fundamental groups.
Still, every homotopy equivalence is homotopic to a diffeomorphism.
In sum, \cref{lem:haut-action-on-base-point} is applicable to both \(M = \Sphere^n\) and \(M = \Torus^n\).

    \begin{proof}[Proof of \cref{thm:exotic-tori}]
    Given an exotic sphere $[\Sigma] \neq [\Sphere^n] \in\manifoldset(\Sphere^n)$, let \(\Sigma \to \Sphere^n\) be any homotopy equivalence (say, equal to the identity on some disk).
    By \cref{prop:connected_sum_injective_on_structure_set}, we obtain an element \([\Torus^n \# \Sigma \to \Torus^n] \neq [\id_{\Torus^n}] \in \structureset(\Torus^n)\).
    Thus \(\Torus^n \# \Sigma\) is not diffeomorphic to \(\Torus^n\) by \cref{lem:haut-action-on-base-point}.
\end{proof}